\documentclass[12pt,reqno]{amsart}

\usepackage{amsthm}
\usepackage{amssymb}
\usepackage{graphicx}
\usepackage{textcomp}
\usepackage{hyperref}
\usepackage{verbatim}        

\usepackage[backend = bibtex,
            natbib = true,
            giveninits = true,
            sorting = nyt,
            style = alphabetic,
            maxcitenames = 99,
            maxbibnames = 99]{biblatex}
\renewbibmacro{in:}{}
\DeclareFieldFormat[article]{volume}{\mkbibbold{#1}}

\renewcommand{\epsilon}{\varepsilon}

\theoremstyle{definition}
\newtheorem{theorem}{Theorem} [section]
\newtheorem{lemma}[theorem]{Lemma}
\newtheorem{corollary}[theorem]{Corollary}

\newtheorem{definition}[theorem]{Definition}

\newcommand{\Q}{\mathbb{Q}}
\newcommand{\R}{\mathbb{R}}
\newcommand{\Z}{\mathbb{Z}}
\newcommand{\C}{\mathbb{C}}
\newcommand{\N}{\mathbb{N}}
\newcommand{\T}{\mathbb{T}}
\newcommand{\Hilbert}{\mathbb{H}}
\newcommand{\Span}{\text{span}}

\newcommand{\Ac}{{\mathcal{A}}}
\newcommand{\I}{\mathcal{I}}

\newcommand{\Oc}{{\mathcal{O}}}

\newcommand{\Plus}{\, + \,}
\newcommand{\plus}{\; + \;}
\newcommand{\Minus}{\, - \,}
\newcommand{\Eq}{\, = \,}
\newcommand{\Ne}{\, \ne \,}
\newcommand{\Ge}{\, \ge \,}

\newcommand{\Le}{\, \le \,}
\newcommand{\Lt}{\, < \,}
\newcommand{\qeddef}{{\quad $\diamondsuit$}}
\newcommand{\takeaway}{\hskip 0.8 pt \backslash \hskip 0.8 pt}
\newcommand{\tinyspace}{\hskip 0.8 pt}
\newcommand{\tinyback}{\hskip -0.8 pt}

\newcommand{\ip}[2]{\langle#1,#2\rangle}
\newcommand{\bigip}[2]{\bigl\langle #1, \, #2 \bigr\rangle}
\newcommand{\Bigip}[2]{\Bigl\langle #1, \, #2 \Bigr\rangle}
\newcommand{\biggip}[2]{\biggl\langle #1, \, #2 \biggr\rangle}
\newcommand{\norm}[1]{\|#1\|}

\newcommand{\bigparen}[1]{\bigl(#1\bigr)}
\newcommand{\Bigparen}[1]{\Bigl(#1\Bigr)}
\newcommand{\biggparen}[1]{\biggl(#1\biggr)}
\newcommand{\set}[1]{\{#1\}}

\newcommand{\ab}{\mathbf{a}}
\newcommand{\bg}{\overline{g}}
\newcommand{\tg}{\widetilde{g}}
\newcommand{\tphi}{\widetilde{\phi}}
\newcommand{\tPhi}{\widetilde{\Phi}}
\newcommand{\CHI}{\hbox{\raise .4ex \hbox{$\chi$}}}
\newcommand{\ub}{\mathbf{u}}
\newcommand{\vb}{\mathbf{v}}
\newcommand{\wb}{\mathbf{w}}

\begin{document}

\title[Overcomplete Reproducing Pairs]{Overcomplete Reproducing Pairs}

\author[L.\ Hart, C.\ Heil, I.\ Katz, and M.\ Northington V]
{Logan Hart, Christopher Heil, Ian Katz, and Michael Northington V}

\address{(L.~Hart) School of Mathematics, Georgia Institute of Technology,
Atlanta, GA 30332, USA}
\email{lhart31@gatech.edu}

\address{(C.~Heil) School of Mathematics, Georgia Institute of Technology,
Atlanta, GA 30332, USA}
\email{heil@math.gatech.edu}

\address{(I.~Katz) School of Mathematics, Georgia Institute of Technology,
Atlanta, GA 30332, USA}
\email{ian.katz95@gmail.com}

\address{(M.~Northington V) School of Mathematics,
Georgia Institute of Technology, Atlanta, GA 30332, USA}
\email{mcnorthington5@gmail.com}

\thanks{Acknowledgement:
This research was partially supported by a grant from the
Simons Foundation.}

\date{September 15, 2023}

\begin{abstract}
The Gaussian Gabor system at the critical density has the property that it is overcomplete in $L^2(\R)$ by exactly one element, and if any single element is removed then the resulting system is complete but is not a Schauder basis.  This paper characterizes systems that are overcomplete by finitely many elements.  Among other results, it is shown that if such a system has a reproducing partner, then it contains a Schauder basis.  While a Schauder basis provides a strong reproducing property for elements of a space, the existence of a reproducing partner only requires a weak type of representation of elements.  Thus for these systems weak representations imply strong representations.  The results are applied to systems of weighted exponentials and to Gabor systems at the critical density.  In particular, it is shown that the Gaussian Gabor system does not possess a reproducing partner.
\end{abstract}

\maketitle

\section{Introduction}
Frames are generalizations of orthonormal bases and Riesz bases.
They were first introduced by Duffin and Schaeffer
in their study of non-harmonic Fourier series \cite{duffin_1952}.
A sequence of vectors $\set{f_n}_{n \geq 0}$ is a
\emph{frame} for a separable, infinite dimensional Hilbert space $\Hilbert$
if there exist constants $A,$ $B > 0$ such that
the following norm equivalence holds:
\begin{align*} 
A\, \norm{f}^2
\Le \sum_{n=0}^{\infty} |\ip{f}{f_n}|^2
\Le B\, \norm{f}^2,
\qquad\text{for every } f \in \Hilbert.
\end{align*}
We refer to $A$ as a \emph{lower frame bound}
and $B$ as an \emph{upper frame bound} for $\set{f_n}_{n \ge 0}.$
A frame is similar to an unconditional basis
in that for every $f \in \Hilbert$ we have an
unconditionally convergent expansion
\begin{align} \label{expansion}
f \Eq \sum_{n=0}^{\infty} c_n f_n
\end{align}
for some scalars $(c_n)_{n \ge 0}.$
However, for a frame the coefficients need not be unique.
If the scalars $c_n$ are unique for every $f,$
then $\set{f_n}_{n \ge 0}$ is a \emph{Riesz basis} for $\Hilbert.$

There are a variety of other notions related to the existence of
representations of elements.
A (strong) \emph{Schauder basis} is a sequence $\set{f_n}_{n \ge 0}$ such that
for each $f \in \Hilbert$ there exist unique scalars $c_n$ such that
equation \eqref{expansion} holds.
A Schauder basis need not have either a lower or upper frame bound.
Further, the representations in equation \eqref{expansion} may
converge conditionally, although there must be a single fixed ordering
of the index set with respect to which the series converge
for every $f.$
A \emph{weak Schauder basis} is similar, except that we only require
that the series in equation \eqref{expansion} converge weakly for every $f.$
However, the \emph{Weak Basis Theorem} implies that every
weak Schauder basis is a Schauder basis.

Every Schauder basis is \emph{exact}, or both minimal and complete.
Complete means that the finite linear span is dense,
while minimal means that no element $f_m$ lies in the closed span
of the other elements $\set{f_n}_{n \ne m}.$
A sequence can be minimal without being exact.
We refer to texts such as \cite{Chr16} and \cite{heil_2011}
for details on frames, Riesz bases, Schauder bases, and related systems.

Given a function $g \in L^2(\R)$ and a countable
index set $\Lambda \subseteq \R^2,$
the \emph{Gabor system} generated by $g$ and $\Lambda$ is
\[
G(g,\Lambda)
\Eq \set{M_\xi T_x g}_{(x,\xi) \in \Lambda}
\Eq \set{e^{2\pi i \xi t} g(t-x)}_{(x,\xi) \in \Lambda},
\]
where $T_x$ is the {translation operator} $T_x g(t) = g(t-x)$
and $M_\xi$ is the {modulation operator}
$M_\xi g(t) = e^{2\pi i \xi t}g(t).$
The compositions $T_x M_\xi$ and $M_\xi T_x$ are
\emph{time-frequency shift operators}.

Usually the index set $\Lambda$ contains some structure.
For example, it may be a lattice $A(\Z^2)$ for some invertible matrix $A,$
or a rectangular lattice $\alpha\Z \times \beta\Z.$
In this paper, we will focus on the case when $\Lambda$
is a rectangular lattice with density $1$ (the \emph{critical density}),
which means that $\alpha\beta = 1.$
By a change of variables, this can always be reduced to the case
$\alpha = \beta = 1.$

The structure of Gabor frames makes them suitable for applications
involving time-dependent frequency content.
Hence, it is not unexpected that Gabor theory has a long history.
Gr{\"o}chenig \cite{grochenig_2001} and Janssen \cite{janssen_2001}
mention that von Neumann \cite{von_2018} claimed (without proof) that,
for the Gaussian atom $\varphi(t) = 2^{1/4}\tinyspace e^{-\pi t^2}$
and the lattice $\Lambda = \Z^2,$
the Gabor system $G(\varphi, \Z^2)$ is complete in $L^2(\R).$
Additionally, Gabor conjectured in \cite{gabor_1946}
that every function $f \in L^2(\R)$
can be represented in the form
\begin{align} \label{gabor expansion}
f \Eq \sum_{k,n \in \Z} c_{nk}(f) \, M_nT_k\varphi,
\end{align}
for some scalars $c_{nk}(f).$
Later, Janssen \cite{janssen_1981} proved that Gabor's conjecture is true,
but with convergence of the series in equation \eqref{gabor expansion}
only in the sense of tempered distributions and not in the norm of $L^2.$
Further, the coefficients $c_{nk}(f)$ may grow with $k$ and $n.$
Gabor's original system $G(\varphi, \Z^2)$ is not a frame.
It is overcomplete by exactly one element (that is, if any single
element of the system is removed then it is still complete, but
if two elements are removed then it is incomplete).
Moreover, the system with one element removed is exact,
but it is not a Schauder basis, and it is not a frame.

Many generalizations of or variations on frames have been introduced.
A \emph{Bessel sequence} need only satisfy the upper frame bound.
A \emph{semi-frame} \cite{antoine_2011, antoine_2012},
need only satisfy one of the two frame bounds.
A \emph{quasibasis} or \emph{Schauder frame} \cite{CDOSZ} is a sequence
$\set{f_n}_{n \ge 0}$ for which there exists a sequence $\set{g_n}_{n \ge 0}$
such that
\begin{equation*} 
f \Eq \sum_{n \geq 0} \, \ip{f}{g_n} \, f_n,
\qquad\text{for every } f \in \Hilbert,
\end{equation*}
where the series converges in norm
with respect to some fixed ordering of the index set.

Recently, Speckbacher and Balazs \cite{speckbacher_2015}
introduced \emph{reproducing pairs} (see also \cite{antoine_2017}).
The general definition is related to
\emph{continuous frames} with respect to arbitrary Borel measures.
However, in this paper we entirely focused on the discrete setting.
In that context, the definition takes the following form.

\begin{definition}
Let $\Psi = \set{\psi_i}_{i \in \I}$ and $\Phi = \set{\phi_i}_{i \in \I}$
be two countable families in $\Hilbert.$
Then $(\Psi, \Phi)$ is a \emph{reproducing pair} for $\Hilbert$
if the operator $S_{\Psi, \Phi} : \Hilbert \to \Hilbert$
that is weakly defined by 
\begin{align} \label{reproducing pair}
\ip{S_{\Psi, \Phi} f}{g}
\Eq \sum_{i \in \I} \ip{f}{\psi_i} \, \ip{\phi_i}{g},
\qquad\text{for } f,\, g \in \Hilbert,
\end{align}
is bounded and boundedly invertible
(that is, a {topological isomorphism} using the terminology
of \cite{heil_2011}).
In this case, we say that $\Psi$ is a reproducing partner for $\Phi,$
and conversely $\Phi$ is a reproducing partner for $\Psi.$
\qeddef\end{definition}

We allow the convergence of the infinite series in
equation \eqref{reproducing pair} to be conditional, in the sense
that there exists some fixed ordering of the index set $\I$ such that the
partial sums of the series converge with respect to that ordering.
Since we will mostly be interested in sequences that are overcomplete
by one element (or finitely many later in the paper),
we will often take the index set to be
$\I = \set{0,1,2,\dots},$ and in that case assume that the
convergence is with respect to the natural ordering.
However, we will make applications to sequences, such as Gabor systems,
that are indexed by other countable sets, and in those settings we will
assume that an ordering has been fixed on the index set.

Since $S_{\Psi, \Phi}^* = S_{\Phi, \Psi},$
if $(\Psi, \Phi)$ is a reproducing pair, then
$(\Psi, S^{-1}_{\Psi, \Phi} \Phi)$ is also a reproducing pair.
Therefore, we can assume without loss of generality that
$S_{\Psi, \Phi} = I$ (see \cite{speckbacher_2015}).
With this assumption, $(\Psi, \Phi)$ is a reproducing pair if
for each $f \in \Hilbert$ we have that the representation
$f = \sum \, \ip{f}{\psi_i} \, \phi_i$ holds weakly, i.e.,
\begin{equation} \label{reproduceidentity_eq}
\ip{f}{g}
\Eq \sum_{i \in \I} \, \ip{f}{\psi_i} \, \ip{\phi_i}{g},
\qquad\text{for all } f,\, g \in \Hilbert.
\end{equation}
In this sense, a reproducing pair is a weak analogue of a
quasibasis or Schauder frame.
Certainly every quasibasis is a reproducing pair; however, it is
unclear to us whether the converse implication holds in general.

In this paper we will consider reproducing pair properties of sequences
that are overcomplete by finitely many elements.
In Section \ref{main section} we consider a sequence $\Phi$ that,
like the original Gabor system $G(\varphi, \Z^2),$
is overcomplete by a single element.
We prove in Theorem \ref{main theorem} that if such a sequence $\Phi$ has a 
reproducing partner $\Psi,$ then it must contain a Schauder basis.
Specifically, the exact sequence obtained by removing that single element
from $\Phi$ is a Schauder basis for $\Hilbert.$

In Section \ref{weighted exponential section}, we focus on systems
of weighted exponentials $\set{e^{2\pi int} g(t)}_{n \in \Z}$
in $L^2[0,1).$
We recover a result from \cite{yoon_2012}, explicitly showing
the existence of weighted exponential systems
that are overcomplete by one element.
We prove in Theorem \ref{exact exponential}
that these systems do not contain a Schauder basis.
By applying Theorem \ref{main theorem} we construct families of sequences
that do not possess a reproducing partner.

We consider Gabor systems in Section \ref{gabor system section}.
Through the use of the Zak transform, we prove in
Theorem \ref{overcomplete gabor} and Corollary \ref{gabor system reproducing}
that there exist families of Gabor sequences at the critical density
that do not possess a reproducing partner.
In particular, we see that the Gaussian Gabor system belongs to this
family.
Balazs and Speckbacher claimed in \cite{speckbacher_2017}
that Gaussian Gabor system $G(\varphi, \Z^2)$
does have a reproducing partner.
However, we demonstrate that this is not possible.

Finally, in Section \ref{generalization} we show how our results
generalize to sequences that are overcomplete by finitely many elements.
These proofs require that we address certain issues of convergence.

\smallskip
\section{Reproducing Pairs and Schauder Bases} \label{main section}

We begin by considering a sequence $\Phi$ that is exact
(both complete and minimal).
A standard fact is that $\Phi = \set{\phi_k}_{k \ge 1}$ is minimal
if and only if there exists a \emph{biorthogonal sequence}
$\tPhi = \set{\tphi_k}_{k \ge 1}$ that satisfies
$\ip{\phi_j}{\tphi_k} = \delta_{jk}.$
Further, an exact sequence has a unique biorthogonal sequence.
(We refer to texts such as \cite{Chr16} or \cite{heil_2011} for details.)

We show first that if an exact sequence has a reproducing partner,
then it is a Schauder basis.

\begin{lemma} \label{exact factorization implies schauder basis}
If an exact sequence $\Phi = \set{\phi_k}_{k \ge 1}$
has a reproducing partner, then $\Phi$ is a Schauder basis for $\Hilbert.$
\end{lemma}
\begin{proof}
Suppose that a reproducing partner $\Psi = \set{\psi_k}_{k \ge 1}$
for $\Phi$ did exist.
Then
\begin{equation} \label{reproducing_eq}
\ip{f}{g}
\Eq \sum_{k \ge 1} \, \ip{f}{\psi_k} \, \ip{\phi_k}{g},
\qquad\text{for all } f,\, g \in \Hilbert.
\end{equation}
Since $\Phi$ is exact, it has a biorthogonal sequence
$\tPhi = \set{\tphi_k}_{k \ge 1}.$
Therefore, by equation \eqref{reproducing_eq} we have for every $f$ that
$$\ip{f}{\tphi_j}
\Eq \sum_{k \ge 1} \ip{f}{\psi_k} \, \ip{\phi_k}{\tphi_j}
\Eq \ip{f}{\psi_j}.$$
Consequently $\psi_j = \tphi_j$ for every $j.$
Therefore $f = \sum \, \ip{f}{\tphi_j} \, \phi_k$ weakly for every $f.$

Now fix $f \in \Hilbert,$ and suppose that $(c_k)_{k \ge 1}$
is a scalar sequence such that
$f = \sum c_k \tinyspace \phi_k$ weakly.
Then
\begin{align*}
\ip{f}{\tphi_j}
& \Eq \sum_{k \ge 1} c_k \tinyspace \ip{\phi_k}{\tphi_j}
\Eq \sum_{k \ge 1} c_k \tinyspace \delta_{jk}
\Eq c_j.
\end{align*}
Hence there is a unique choice of coefficients for which we have
$f = \sum c_k \tinyspace \phi_k$ weakly.
Therefore $\Phi$ is a weak Schauder basis for $\Hilbert,$ and so,
by the Weak Basis Theorem (see \cite[Thm.~4.30]{heil_2011}),
$\Phi$ is a strong Schauder basis for $\Hilbert.$
\end{proof}

Now we prove that if a sequence $\Phi$ that is overcomplete by one element
possesses a reproducing partner, then it must contain a Schauder basis.

\begin{theorem} \label{main theorem}
Assume that $\Phi = \set{\phi_k}_{k \geq 0}$
satisfies the following properties.
\begin{enumerate}
\item[\textup{(a)}]
$\Phi' = \set{\phi_k}_{k \geq 1}$ is exact in $\Hilbert.$

\smallskip
\item[\textup{(b)}]
$\Phi$ has a reproducing partner $\Psi = \set{\psi_k}_{k \ge 0}.$
\end{enumerate}
Then $\Phi'$ is a Schauder basis for $\Hilbert.$
\end{theorem}
\begin{proof}
Since $\Phi'$ is exact,
it has a biorthogonal sequence $\tPhi = \set{\tphi_k}_{k \geq 1}.$
Further, since $(\Psi,\Phi)$ is a reproducing pair,
equation \eqref{reproduceidentity_eq} holds with $\I = \set{0, 1, 2, \dots}.$
If $j \geq 1,$ then we have for all $f \in \Hilbert$ that
\begin{align*}
\ip{f}{\tphi_j}
\Eq \sum_{k=0}^{\infty} \ip{f}{\psi_k} \, \ip{\phi_k}{\tphi_j}
& \Eq \ip{f}{\psi_0} \, \ip{\phi_0}{\tphi_j} \plus
      \sum_{k=1}^{\infty} \ip{f}{\psi_k} \, \ip{\phi_k}{\tphi_j}
      \\
& \Eq \Bigip{f}{\ip{\tphi_j}{\phi_0} \, \psi_0} \plus \ip{f}{\psi_j}
      \qquad\text{\footnotesize (by biorthogonality)}
      \\[1 \jot]
& \Eq \Bigip{f}{\ip{\tphi_j}{\phi_0} \, \psi_0 \Plus \psi_j}.
\end{align*}
Therefore
\begin{equation} \label{tphi_eq}
\tphi_j \Eq \ip{\tphi_j}{\phi_0} \, \psi_0 \Plus \psi_j,
\qquad\text{for every } j \ge 1.
\end{equation}

Now assume that $\psi_0 \ne 0,$ as otherwise the result follows trivially.
Since $\Phi'$ is complete, there is some $n \geq 1$ such that
$\ip{\phi_n}{\psi_0} \ne 0.$
Using equation \eqref{tphi_eq} to substitute for $\psi_k,$
we compute that if $g \in \Hilbert$ then
\begin{align*}
\ip{\phi_n}{g}
& \Eq \sum_{k=0}^{\infty} \, \ip{\phi_n}{\psi_k} \, \ip{\phi_k}{g}
      \\
& \Eq \ip{\phi_n}{\psi_0} \, \ip{\phi_0}{g} \plus
      \sum_{k=1}^{\infty} \, \ip{\phi_n}{\psi_k} \ip{\phi_k}{g}
      \allowdisplaybreaks \\
& \Eq \ip{\phi_n}{\psi_0} \, \ip{\phi_0}{g} \plus \sum_{k=1}^{\infty} \,
      \Bigip{\phi_n}{\tphi_k - \ip{\tphi_k}{\phi_0} \, \psi_0} \,
      \ip{\phi_k}{g}
      \qquad\text{\footnotesize (by equation \eqref{tphi_eq})}
      \allowdisplaybreaks \\
& \Eq \ip{\phi_n}{\psi_0} \ip{\phi_0}{g} \plus \sum_{k=1}^{\infty} \,
      \Bigparen{\ip{\phi_n}{\tphi_k} - \ip{\phi_0}{\tphi_k} \,
      \ip{\phi_n}{\psi_0}} \, \ip{\phi_k}{g}
       \allowdisplaybreaks \\
& \Eq \ip{\phi_n}{\psi_0} \, \ip{\phi_0}{g} \plus \sum_{k=1}^{\infty} \,
      \Bigparen{\delta_{kn} - \ip{\phi_0}{\tphi_k} \, \ip{\phi_n}{\psi_0}} \,
      \ip{\phi_k}{g}
      \\
& \Eq \ip{\phi_n}{\psi_0} \, \ip{\phi_0}{g} \plus \ip{\phi_n}{g} \Minus
      \sum_{k=1}^{\infty} \ip{\phi_0}{\tphi_k} \,
      \ip{\phi_n}{\psi_0} \ip{\phi_k}{g}.
\end{align*}
Therefore,
\[
\ip{\phi_n}{\psi_0} \, \ip{\phi_0}{g}
\Eq \ip{\phi_n}{\psi_0} \sum_{k=1}^{\infty} \,
    \ip{\phi_0}{\tphi_k} \, \ip{\phi_k}{g}.
\]
Since $\ip{\phi_n}{\psi_0} \ne 0,$ we can cancel that factor and conclude that
\begin{equation} \label{phi0_eq}
\phi_0 \Eq \sum_{k=1}^{\infty} \, \ip{\phi_0}{\tphi_k} \, {\phi}_k
\quad\text{weakly}.
\end{equation}

Now we will show that $\Phi'$ is a Schauder basis.
If $g,$ $h \in \Hilbert,$ then
\begin{alignat*}{2}
\ip{g}{h}
& \Eq \ip{g}{\psi_0} \, \ip{\phi_0}{h} \plus
      \sum_{k=1}^{\infty} \, \ip{g}{\psi_k} \ip{\phi_k}{h}
      && \qquad\text{\footnotesize (reproducing property)}
      \\
& \Eq \ip{g}{\psi_0} \sum_{k=1}^{\infty} \,
      \ip{\phi_0}{\tphi_k} \,  \ip{\phi_k}{h} \plus
      \sum_{k=1}^{\infty} \ip{g}{\psi_k} \, \ip{\phi_k}{h}
      && \qquad\text{\footnotesize (by equation \eqref{phi0_eq})}
      \allowdisplaybreaks \\
& \Eq \sum_{k=1}^{\infty} \,
      \Bigparen{\tinyback \ip{g}{\psi_0} \, \ip{\phi_0}{\tphi_k} \Plus
      \ip{g}{\psi_k} \tinyback} \, \ip{\phi_k}{h}
      \allowdisplaybreaks \\
& \Eq \sum_{k=1}^{\infty} \,
      \Bigip{g}{\ip{\tphi_k}{\phi_0} \, \psi_0 \Plus \psi_k} \, \ip{\phi_k}{h}
      \\
& \Eq \sum_{k=1}^{\infty} \, \ip{g}{\tphi_k} \, \ip{\phi_k}{h}.
      && \qquad\text{\footnotesize (by equation \eqref{tphi_eq})}
\end{alignat*}
Therefore $(\tPhi,\Phi')$ is a reproducing pair, so
Lemma \ref{exact factorization implies schauder basis}
implies that $\Phi'$ is a Schauder basis for $\Hilbert.$
\end{proof}

\section{Weighted Exponentials and Reproducing Partners}
\label{weighted exponential section}

\subsection{Background} 

In this section we will apply our results to sequences of
weighted exponentials in $L^2(\T),$ where $\T = [0,1).$
These have the form $E(g,\Z) = \set{g\tinyspace e_n}_{n \in \Z},$
where $e_n(t) = e^{2\pi i n t}$
(equivalently, we could consider the trigonometric system
$\set{e_n}_{n \in \Z}$ in a weighted $L^2$ space).
A characterization of when $E(g,\Z)$ is
complete, minimal, exact, a frame, an unconditional basis,
or an orthonormal basis can be found in textbooks
such as \cite{Chr16} or \cite{heil_2011}.
A much deeper classical result due to
Hunt, Muckenhoupt, and Wheeden \cite{HMW73}
is that $E(g,\Z)$ is a Schauder basis for $L^2(\T)$ with respect to the ordering
$\Z = \set{0,-1,1,-2,2,\dots}$ if and only if
$|g|^2$ is an \emph{$\Ac_2(\T)$ weight}.

However, those results apply when the index set is
the full set of integers.
We are interested in systems that are overcomplete by one or finitely many
elements, and hence we deal with subsequences $E(g,\Lambda)$
that are indexed by a proper subset $\Lambda$ of $\Z.$
Kazarian \cite{kazarian_2014} characterized the functions $g$
and index sets $\Lambda$ such that $E(g,\Lambda)$ is
complete or minimal in $L^p(\T)$ for $1 \leq p < \infty.$
Additional related results are in
\cite{kazarian_2017, kazarian_2018, kazarian_2019}.
A characterization of functions $g$ and finite sets $F \subseteq \Z$
such that $E(g, \Z \takeaway F)$ is exact in $L^2(\T)$
was given by Heil and Yoon \cite{yoon_2012}.
Recently, Zikkos \cite{zikkos_2022} proved that
the closed span in $L^2(\gamma,\beta)$ of the system
$E(t^k,\Lambda)
= \set{t^k e^{\lambda_n t} :\, n \in \N, \, k = 1, 2, \ldots, \mu_n-1},$
with $\mu_i \in \N,$ is equal to the
closed span of its unique biorthogonal sequence
$r_{\Lambda} = \set{r_{n,k} :\, n \in \N,\, k = 1, 2, \ldots, \mu_n-1}$
if some constraints on $\Lambda$ and the $\mu_i$ are satisfied.
Further, in this case each $f \in L^2(\gamma,\beta)$
admits a Fourier-like series representation
\[
f(t)
\Eq \sum_{n=1}^{\infty} \, \biggparen{\sum_{k=1}^{\mu_n-1}
    \ip{f}{r_{n,k}} \, t^k \!} \, e^{\lambda_n t},
\]
where the series converges uniformly
on closed subintervals of $(\gamma,\beta).$

\subsection{Applications}

The following theorem combines a characterization from \cite{yoon_2012}
with a basis result from \cite{shukurov_2018}.
We include the proof for completeness.

\begin{theorem} \label{exact exponential}
Let $g \in L^2(\T)$ be such that $1/g \notin L^2(\T).$
If $(t-t_0)/g(t) \in L^2(\T)$
for some $t_0 \in \T,$ then
$E(g, \Z \takeaway \set{k})$
is exact for every $k \in \Z.$
However, there is no ordering of $\Z \takeaway \set{k}$
such that $E(g, \Z \takeaway \set{k})$
is a Schauder basis for $L^2(\T).$
\end{theorem}
\begin{proof}
Fix $k \in \Z.$
We will first show that $E(g, \Z \takeaway \set{k})$ is complete.
If $f \in L^2(\T)$ satisfies $\ip{f}{g\tinyspace e_n} = 0$ for all $n \ne k,$
then $\ip{f \tinyspace \bg}{e_n} = 0$ for every $n \ne k.$
Since $f\overline{g}\in L^1(\T)$ and functions in $L^1(\T)$ are determined by their Fourier coefficients,
it follows that $f \tinyspace \bg = c \tinyspace e_k$ for some constant $c,$
and hence $f e_{-k} = c/\bg.$
Since $1/g \notin L^2(\T),$ we must have $c = 0.$
Therefore $E(g, \Z \takeaway \set{k})$ is complete.

Next, for each $n \ne k$ let $c_n = -e^{2\pi i(n-k)t_0},$
so that $e_n + c_n e_k$ vanishes at $t_0.$
Then the function
\[
\tg_n \Eq \frac{e_n + c_n e_k}{\bg}
\]
belongs to $L^2(\T),$ and
$\ip{g\tinyspace e_m}{\tg_n}
= \ip{e_m}{e_n + c_n e_k}
= \delta_{mn}$
for $m,$ $n \ne k.$
Therefore $\set{\tg_n}_{n \ne k}$ is biorthogonal to
$E(g, \Z \takeaway \set{k}),$ so this sequence is minimal.

Now we will show that $E(g, \Z \takeaway \set{k})$ is not a Schauder basis.
Assume that there were some ordering of $\Z \takeaway \set{k}$
such that $E(g, \Z \takeaway \set{k})$
formed a Schauder basis for $L^2(\T).$
Then there would exist unique coefficients $d_n$ such that
\begin{equation} \label{gexpansion_eq}
g\tinyspace e_k \Eq \sum_{n \ne k} d_n\tinyspace g\tinyspace e_n,
\end{equation}
where this sum converges in norm with respect to the specified
ordering of $\Z \takeaway \set{k}.$
Using the biorthogonality established earlier, it follows that
if $m \ne k$ then
\[
\ip{g\tinyspace e_k}{\tg_m}
\Eq \sum_{n \ne k} d_n\tinyspace \ip{g\tinyspace e_n}{\tg_m}
\Eq \sum_{n \ne k} d_n\tinyspace \delta_{nm}
\Eq d_m.
\]
However, if $m \ne k$ then we also have that
\[
\ip{g\tinyspace e_k}{\tg_m}
\Eq \Bigip{g\tinyspace e_k}{\frac{e_m + c_m e_k}{\bg}}
\Eq \ip{e_k}{e_m + c_m e_k}
\Eq \delta_{km} + \overline{c_m}
\Eq \overline{c_m}.
\]
Therefore $d_m = \overline{c_m}$ for $m \ne k.$
Consequently, since the series in equation \eqref{gexpansion_eq}
converges in norm, we must have
$\norm{\overline{c_n}\tinyspace g\tinyspace e_n}_2 \to 0$ as $n\to \infty.$
But $|c_n| = 1,$ so
$\norm{\overline{c_n}\tinyspace g\tinyspace e_n}_2
\Eq \norm{g}_2$
for every $n,$ which is a contradiction.
\end{proof}

Using Theorem \ref{main theorem},
we obtain the following corollary.

\begin{corollary} \label{weighted exponental reproducing}
Let $g \in L^2(\T)$ be such that $1/g \notin L^2(\T).$
If there exists some point $t_0 \in \T$ such that $(t-t_0)/g(t) \in L^2(\T),$
then there does not exist any sequence $\Psi \subseteq L^2(T)$
such that $\bigparen{\Psi, E(g,\Z)}$ is a reproducing pair.
\end{corollary}
\begin{proof}
By Theorem \ref{exact exponential}, the sequence $E(g,\Z)$
is overcomplete by one element,
and if we remove any element then the resulting sequence
is not a Schauder basis.
Theorem \ref{main theorem} therefore implies that there is no
sequence $\Psi \subseteq L^2(\T)$ such that
$\bigparen{\Psi, E(g,\Z)}$ is a reproducing pair.
\end{proof}

For example, the function $g(t) = t$ satisfies the hypotheses of
Corollary \ref{weighted exponental reproducing},
so the system $E(t,\Z)$ does not possess a reproducing partner.

\smallskip
\section{Gabor Systems and Reproducing Pairs} \label{gabor system section}

\subsection{Background}
Von Neumann's claim that the Gaussian Gabor system
at the critical density is complete was proven independently by
Perelomov \cite{perelomov_1971},
Bargmann, Butera, Girardello, and Klauder \cite{bargmann_1971},
and Bacry, Grossmann, and Zak \cite{bacry_1975}.
It was later conjectured by Daubechies and Grossmann that
$G(\varphi, \alpha \Z \times \beta \Z)$ is a frame if and only if
$0 < \alpha \beta < 1$ \cite{daubechies_1988, daubechies_1990}.
This conjecture was proven in full by Lyubarskii \cite{lyubarskii_1992}
and by Seip and Wallst{\'e}n \cite{seip_1992, seip_1992_2}.
At the critical density, Folland \cite{folland_1989} proved that
$G(\varphi, \alpha\Z \times 1/\alpha \Z)$ is not a frame,
and is overcomplete by exactly one element.
He further showed that if any single element is removed,
the resulting system is exact but not a Schauder basis.
We include an alternative proof of this result in
Theorem \ref{overcomplete gabor}.
On the topic of overcompleteness, it was shown in \cite{balan_2003}
that every Gabor frame $G(g, \alpha\Z \times \beta\Z),$
with $\alpha\beta < 1,$ has infinite excess
(so is overcomplete by infinitely many elements).
Gr{\"o}chenig and St{\"o}cker \cite{grochenig_2013} showed that if
$g \in L^2(\R)$ is a {totally positive function of finite type}
(which includes the Gaussian function $\varphi$), then the Gabor system
$G(g, \alpha\Z \times\beta \Z)$ is a frame for $L^2(\R)$
if an only if $\alpha \beta < 1$. 
Recently, Gr{\"o}chenig proved \cite{grochenig_2023} that if
$g \in L^1(\R)$ is totally positive and $\alpha\beta \in \Q,$
then $G(g, \alpha\Z \times \beta\Z)$
is a frame for $L^2(\R)$ if and only if $\alpha\beta < 1.$

The Zak transform is an important tool for studying Gabor systems
at the critical density $\alpha\beta = 1$
(which we reduce to $\alpha = \beta = 1$ by a change of variables).
Gr{\"o}chenig \cite{grochenig_2001} remarks that the Zak transform
was first introduced by Gel{\textquotesingle}fand \cite{gelfand_1950}.
As with many useful tools, it has been rediscovered numerous times
and goes by a variety of names.
Weil \cite{weil_1964} defined a Zak transform for
locally compact abelian groups, and this transform is often called the
\emph{Weil-Brezin} map in representation theory
and abstract harmonic analysis \cite{schempp_1984}.
Zak rediscovered this transform, which he called the $k$-$q$ transform,
in his work on quantum mechanics \cite{zak_1967}.
The terminology ``Zak transform'' has become customary in
applied mathematics and signal processing.
We refer to texts such as \cite{Chr16, grochenig_2001, heil_2011}
for details on the Zak transform.

Let $Q = [0,1)^2.$
The \emph{Zak transform} is the unitary map $Z : L^2(\R) \to L^2(Q)$
defined by
\begin{equation} \label{zak_eq}
Zf(x,\xi) \Eq \sum_{j \in \Z} f(x-j) \, e^{2\pi i j \xi},
\qquad\text{for } (x,\xi) \in Q.
\end{equation}
The series in equation \eqref{zak_eq} converges unconditionally
in the norm of $L^2(Q).$
Since $Z$ is unitary, it preserves properties such as completeness,
minimality, being a frame, being a Riesz basis, and so forth.

If $g \in L^2(\R),$ then the Gabor system generated by $g$ at the
critical density is
$G(g,\Z^2) = \set{M_n T_k g}_{k,n \in \Z}.$
For $k,$ $n \in \Z,$ let
\[
E_{nk}(x,\xi) = e^{2\pi inx} \, e^{-2\pi ik\xi},
\qquad\text{for } (x,\xi) \in \R^2.
\]
The Zak transform has the property that
\[
Z(M_nT_k g)(x,\xi)
\Eq E_{nk}(x,\xi) \, Zg(x,\xi)
\Eq e^{2\pi inx} \, e^{-2\pi ik\xi} \, Zg(x,\xi).
\]
Consequently, the image of the Gabor system under $Z$ is
\[
Z\bigparen{G(g,\Z^2)}
\Eq \set{E_{nk} \tinyspace Zg}_{(k,n) \in \Z^2}.
\]
This is a two-dimensional version of the systems of weighted exponentials
that we studied in Section \ref{weighted exponential section}.
Thus we expect that similar results will hold, although there are some
issues due to the higher-dimensional setting.

\subsection{Applications}

We will need the following lemma regarding the existence
of certain functions in $L^2(Q).$
We will use the following notation for a ``cone function'' centered at
$(x_0,\xi_0)$:
\[
\rho_{x_0,\xi_0}(x,\xi)
\Eq \sqrt{(x-x_0)^2 \Plus (\xi-\xi_0)^2}.
\]

\begin{lemma} \label{exponential bound}
Fix $(x_0,\xi_0) \in Q$ and $(a,b) \in \Z^2.$
For $(n,k) \ne (a,b),$ let $c_{nk}$ be the scalar of unit modulus such that
$E_{nk}(x_0,\xi_0) + c_{nk}\tinyspace E_{ab}(x_0,\xi_0) = 0.$
Then
\[
\frac{E_{nk} \Plus
c_{nk}\tinyspace E_{ab}}{\rho_{x_0,\xi_0}}
\]
is bounded on $Q \takeaway \set{(x_0,\xi_0)},$ and hence belongs to $L^2(Q).$
\end{lemma}

\begin{proof}
If $(x,\xi) \in Q$ and $(n,k) \ne (a,b),$ then a direct calculation shows that
\begin{align*}
|E_{nk}(x,\xi) \Plus c_{nk}\tinyspace E_{ab}(x,\xi)|
\Eq |E_{n-a,k-b}(x-x_0,\xi-\xi_0) - 1|.
\end{align*}
Therefore it suffices to show that if $(n,k) \ne (0,0)$ then
\begin{equation} \label{Enk_eq}
\frac{E_{nk}(x-x_0,\xi-\xi_0) \Minus 1}{\rho_{x_0,\xi_0}(x,\xi)}
\quad\text{is bounded on } Q \takeaway \set{(x_0,\xi_0)}.
\end{equation}
First we note that
\begin{align*}
|E_{nk}(x,\xi) - 1|
\Eq |e^{2\pi i (nx - k\xi)} - 1|
& \Le 2\pi\tinyspace |nx - k\xi|
      \\[0.5 \jot]
& \Le 2\pi\tinyspace \bigparen{|nx| + |k\xi|}
      \\[0.5 \jot]
& \Le 2\pi \, \sqrt{k^2 + n^2} \, \sqrt{x^2 + \xi^2}.
\end{align*}
Therefore, for $(n,k) \ne (0,0)$ we have that
\begin{align*}
|E_{nk}(x-x_0,\xi-\xi_0) - 1|
& \Le 2\pi \, \sqrt{k^2 + n^2} \; \rho_{x_0,\xi_0}(x,\xi),
\end{align*}
and equation \eqref{Enk_eq} follows from this.
\end{proof}

Now, we prove an analogue of Theorem \ref{exact exponential}
for Gabor systems at the critical density.

\begin{theorem} \label{overcomplete gabor}
Let $g \in L^2(\R)$ be such that $1/Zg \notin L^2(Q).$
If there is some point $(x_0,\xi_0) \in Q$ such that
$\rho_{x_0,\xi_0}/Zg \in L^2(Q),$ then
the Gabor system $G\bigparen{g,\Z^2 \takeaway \set{(a,b)}}$
is exact in $L^2(\R)$ for every pair $(a,b) \in \Z^2.$
However, there is no ordering of $\Z^2 \takeaway \set{(a,b)}$
for which this system is a Schauder basis for $L^2(\R).$
\end{theorem}
\begin{proof}
The proof is similar to the proof of Theorem \ref{exact exponential},
so we only sketch the details.
Since $\ip{f}{M_nT_kg}_{L^2(\R)} = \ip{Zf}{E_{nk}Zg}_{L^2(Q)},$
completeness follows immediately from the assumption
that $1/Zg \notin L^2(Q).$

Observe that
\[
\frac{E_{nk} \Plus c_{nk}\tinyspace E_{ab}}{\overline{Zg}}
\Eq \frac{E_{nk} \Plus c_{nk}\tinyspace E_{ab}}{\rho_{x_0,\xi_0}} \cdot
    \frac{\rho_{x_0,\xi_0}}{\overline{Zg}}.
\]
This is the product of a bounded function with a square-integrable
function, so belongs to $L^2(Q).$
Therefore, since $Z$ is unitary, there is a function
$\tg_{nk} \in L^2(\R)$ such that
\[
Z(\tg_{nk}) \Eq \frac{E_{nk} \Plus c_{nk}\tinyspace E_{ab}}{\overline{Zg}}.
\]
Since
$\ip{M_nT_kg}{\tg_{nk}}_{L^2(\R)}
= \bigip{E_{nk}\tinyspace Zg}{Z(\tg_{nk})}_{L^2(Q)},$
it follows that $\set{\tg_{nk}}_{(n,k) \ne (a,b)}$ is biorthogonal to
$G\bigparen{g, \Z^2 \takeaway \set{(a,b)}}.$
Therefore that system is minimal.

Finally, by again using the fact that $Z$ is unitary, the proof that
$G\bigparen{g, \Z^2 \takeaway \set{(a,b)}}$ is not a Schauder basis
is very similar to the argument presented in the proof of
Theorem \ref{exact exponential}.
\end{proof}

Next we give the Gabor system equivalent of
Corollary \ref{weighted exponental reproducing}.

\begin{corollary} \label{gabor system reproducing}
Let $g \in L^2(\R)$ be such that $1/Zg \notin L^2(Q).$
If there is some point $(x_0,\xi_0) \in Q$ such that
$\rho_{x_0,\xi_0}/Zg \in L^2(Q),$
then $G(g,\Z^2)$ does not possess a reproducing partner.
\end{corollary}
\begin{proof}
By Theorem \ref{overcomplete gabor},
the Gabor system $G(g,\Z^2)$ is overcomplete by one element,
and if we remove any one element then the resulting sequence
is not a Schauder basis.
Theorem \ref{main theorem} therefore implies that
$G(g,\Z^2)$ does not have a reproducing partner.
\end{proof}

\subsection{The Original Gabor System}

We will show that the Gaussian Gabor system at the critical density
does not possess a reproducing partner.
We set $\varphi(t) = 2^{1/4}\tinyspace e^{-\pi t^2},$
and let $\Theta = Z \varphi$ be the Zak transform of the Gaussian function.
Because $\varphi$ is smooth and decays quickly, $\Theta$ is smooth
on $Q,$ and furthermore it has a single zero in $Q,$ at the point $(1/2,1/2).$
In fact, we have explicitly (compare \cite{janssen_2006}) that
\begin{align}
\Theta(x,\xi)
& \Eq 2^{1/4} \sum_{k \in \Z} e^{-\pi(x-k)^2} \, e^{2\pi i k\xi}
      \notag \\
& \Eq -2^{1/4} i \,  e^{-\pi (x-1/2)^2 + \pi i (\xi-1/2)^2} \,
      \theta_1\bigparen{\pi\bigparen{\xi - \tfrac{1}{2} -
      i \bigparen{x - \tfrac{1}{2}}}, \, e^{-\pi}},
      \notag 
\end{align}
where $\theta_1$ is the first Jacobi theta function,
\begin{equation*} 
\theta_1(z,q)
\Eq -i\, \sum_{k \in \Z} (-1)^k \, q^{(k+1/2)^2} \, e^{(2k+1)iz},
\end{equation*}
see \cite[Chap.~1]{whittaker_1927}.
We will implicitly assume henceforth that $q = e^{-\pi},$
and just write $\theta_1(z)$ instead of $\theta_1(z,q).$

\begin{corollary}
$1/\Theta \notin L^2(Q),$ but
$\dfrac{\rho_{1/2,1/2}}{\Theta} \,\in\, L^2(Q).$
Consequently, $G(\varphi,\Z^2)$ does not possess a reproducing partner.
\end{corollary}
\begin{proof}
The Taylor series expansion of $\Theta$ about the point $(1/2,1/2)$ is
\begin{align*}
\Theta(x,\xi)
& \Eq -2^{1/4}\pi\, \theta_1^{\,\prime}(0) \,
  \Bigparen{\tinyback\bigparen{x -  \tfrac{1}{2}} \Plus
  i \bigparen{\xi -  \tfrac{1}{2}} \tinyback} \plus
  \Oc\Bigparen{\tinyback\bigparen{x- \tfrac{1}{2}}^2 \Plus
  \bigparen{\xi -  \tfrac{1}{2}}^2},
\end{align*}
where 
\[
\theta_1^{\,\prime}(0)
\Eq \sum_{k \in \Z} \, (-1)^k \, (2k+1) \, e^{-\pi(k + 1/2)^2}
\Eq 2 \, \sum_{k=0}^{\infty} \, (-1)^k \, (2k+1) \, e^{-\pi(k + 1/2)^2}
\Ne 0.
\]
Therefore, there exist constants $C > 0$ and $0 < \delta < 1/2$ such that
\[
|\Theta(x,\xi)|
\Ge C \,
    \sqrt{\bigparen{x- \tfrac{1}{2}}^2 \Plus \bigparen{\xi -  \tfrac{1}{2}}^2}
\Eq C \, \rho_{1/2,1/2}(x,\xi),
\qquad\text{if } (x,\xi) \in B_\delta,
\]
where $B_\delta$ is the open ball of radius $\delta$ centered at $(1/2,1/2).$
Additionally, since $\Theta$ is continuous and its only zero in $Q$
is at the point $(1/2,1/2),$ there is some $c > 0$ such that
\[
|\Theta(x,\xi)| \Ge c,
\qquad\text{if } (x,\xi) \in Q \takeaway B_\delta.
\]
Hence, we compute that
\begin{align*}
\iint_Q \, \frac{\rho_{1/2,1/2}(x,\xi)^2}{|\Theta(x,\xi)|^2} \,dx \, d\xi
& \Eq \iint_{Q \takeaway B_\delta}
      \frac{\rho_{1/2,1/2}(x,\xi)^2}{|\Theta(x,\xi)|^2} \, dx \, d\xi \plus
      \iint_{B_\delta} \frac{\rho_{1/2,1/2}(x,\xi)^2}{|\Theta(x,\xi)|^2} \,
      dx\,d\xi
      \\
& \Le \iint_{Q \takeaway B_\delta} \frac{\rho_{1/2,1/2}(x,\xi)^2}{c^2} \,
      dx \, d\xi \plus \iint_{B_\delta} \frac{1}{C^2} \, dx \, d\xi
      \\
& \Le \frac{1}{2c^2} \Plus \frac{\pi\delta^2}{C^2}
\Lt \infty.
\end{align*}
This shows that $\rho_{1/2,1/2}/\Theta$ is square-integrable on $Q.$
A similar argument can be used to prove that $1/\Theta \notin L^2(Q).$
Theorem \ref{overcomplete gabor} therefore implies that
$G(\varphi,\Z^2)$ is overcomplete by exactly one element,
and if any one element is removed then the resulting system is
not a Schauder basis for $L^2(\R).$
Thus, it follows from Corollary \ref{gabor system reproducing}
that $G(\varphi,\Z^2)$ does not possess a reproducing partner.
\end{proof}

\section{Overcomplete by Finitely Many Elements} \label{generalization}

In this section we will generalize Theorem \ref{main theorem}
to sequences that are overcomplete by $n > 1$ elements.

\subsection{Lemmas}

The following lemma will allow us to reduce to the case where the
sets of overcomplete elements $\set{\phi_0, \ldots, \phi_{n-1}}$
and $\set{\psi_0,\ldots, \psi_{n-1}}$ are each linearly independent.

\begin{lemma} \label{reduction lemma} 
Let $\phi = \set{\phi_0, \ldots, \phi_{n-1}}$
and $\psi = \set{\psi_0,\ldots, \psi_{n-1}}$ be subset of $\Hilbert.$
If either $\phi$ or $\psi$ is linearly dependent, then there exist elements
$\psi_k',$ $\phi_k' \in \Hilbert$ such that
\begin{equation} \label{reduction_eq}
\sum_{k=0}^{n-1} \, \ip{f}{\psi_k} \, \ip{\phi_k}{g}
\Eq \sum_{k=0}^{n-2} \, \ip{f}{\psi_k'} \, \ip{\phi_k'}{g},
\qquad\text{for all } f,\, g \in \Hilbert.
\end{equation}
Further, we can choose these functions so that
$\phi_j' \in \underset{k = 0,\ldots, n-2}{\Span} \set{\phi_k}$
and $\psi_j' \in \underset{k = 0,\ldots, n-2}{\Span} \set{\psi_k}.$
\end{lemma}
\begin{proof}
Without loss of generality, assume that $\psi$ is linearly dependent,
so there exist coefficients $c_j$ such that
$\psi_{n-1} \Eq \sum_{j=0}^{n-2} c_j\tinyspace \psi_j.$
If we set $\psi_k' = \psi_k$ and
$\phi_k' = \phi_k + \overline{c_k}\tinyspace \phi_{n-1}$
for $k = 0, \dots, n-2,$ then equation \eqref{reduction_eq} holds.
\end{proof}

We also need a lemma that if $\set{\phi_m}_{m \in \I}$
is complete and $\set{\psi_0, \ldots, \psi_{n-1}}$ is linearly independent
in $\Hilbert,$ then we can create a particular sequence of vectors
that is complete in $\C^n.$

\begin{lemma} \label{linear independence completeness}
Assume vectors $\psi_0, \ldots, \psi_{n-1}$
are linearly independent in $\Hilbert,$
and $\set{\phi_m}_{m \geq n}$ is complete in $\Hilbert.$
Set
\[
\vb_m
\Eq \begin{bmatrix}
    \ip{\psi_0}{\phi_m} \\ \vdots \\ \ip{\psi_{n-1}}{\phi_m}
    \end{bmatrix},
\qquad\text{for } m \ge n.
\]
Then $\set{\vb_m}_{m \ge n}$ is complete in $\C^n,$ and hence spans $\C^n.$
\end{lemma}
\begin{proof}
Then for each $m \ge n$ we have that
\[
\biggip{\sum_{j=0}^{n-1} \overline{a_j}\tinyspace \psi_j}{\phi_m}
\Eq \sum_{j=0}^{n-1}  \, \overline{a_j}\,\ip{\psi_j}{\phi_m}
\Eq \vb_m \cdot \ab 
\Eq 0.
\]
Since $\set{\phi_m}_{m \geq n}$ is complete, we therefore have
$\sum_{j=0}^{n-1} \overline{a_j}\tinyspace \psi_j = 0.$
But $\{\psi_0, \ldots, \psi_{n-1}$\} is linearly independent,
so $a_j = 0$ for $j=0,\ldots, n-1,$ and thus $\ab = 0.$
Since $\C^n$ is finite-dimensional, it follows that the
finite linear span of $\set{v_m}_{m \ge n}$ is $\C^n.$
\end{proof}

\subsection{Generalization of Theorem \ref{main theorem}}

Now we consider sequences that are overcomplete by $n > 1$ elements.

\begin{theorem} \label{overcomplete by n}
Assume that $\Phi = \set{\phi_k}_{k \geq 0}$
satisfies the following properties.
\begin{enumerate}
\item[\textup{(a)}]
$\Phi' = \set{\phi_k}_{k \geq n}$ is exact in $\Hilbert.$

\smallskip
\item[\textup{(b)}]
$\Phi$ has a reproducing partner $\Psi = \set{\psi_k}_{k \ge 0}.$
\end{enumerate}
Then $\Phi'$ is a Schauder basis for $\Hilbert.$
\end{theorem}
\begin{proof}
By repeatedly applying Lemma \ref{reduction lemma} if necessary,
we can assume that $\set{\phi_0,\ldots, \phi_{n-1}}$
and $\set{\psi_0, \ldots, \psi_{n-1}}$ are each linearly independent.

If $j \geq n,$ then we have for all $f \in \Hilbert$ that
\begin{alignat*}{2}
\ip{f}{\tphi_j}
& \Eq \sum_{k=n}^{\infty} \, \ip{f}{\psi_k} \, \ip{\phi_k}{\tphi_j} \plus
      \sum_{k=0}^{n-1} \, \ip{f}{\psi_k} \, \ip{\phi_k}{\tphi_j}
      && \qquad\text{\footnotesize (reproducing property)}
      \\
& \Eq \ip{f}{\psi_j} \Plus
      \sum_{k=0}^{n-1} \, \ip{f}{\psi_k} \, \ip{\phi_k}{\tphi_j}
      && \qquad\text{\footnotesize (biorthogonality)}
      \\
& \Eq \Bigip{f}{\psi_j \Plus \sum_{k=0}^{n-1} \,
      \ip{\tphi_j}{\phi_k} \, \psi_k}.
\end{alignat*}
Therefore
\begin{equation} \label{tphij_eq}
\tphi_j \Eq \psi_j \Plus \sum_{k=0}^{n-1} \, \ip{\tphi_j}{\phi_k} \, \psi_k,
\qquad\text{for every } j \ge n.
\end{equation}
Hence, if $m \geq n$ and $g \in \Hilbert,$ then
\begin{alignat}{2}
\ip{\phi_m}{g}
& \Eq \sum_{k=n}^{\infty} \, \ip{\phi_m}{\psi_k} \, \ip{\phi_k}{g} \plus
      \sum_{k=0}^{n-1} \, \ip{\phi_m}{\psi_k} \, \ip{\phi_k}{g}
      \quad\text{\footnotesize (reproducing property)}
      \notag \\
& \Eq \sum_{k=n}^{\infty} \, \Bigip{\phi_m}{\tphi_k -
      \sum_{j=0}^{n-1} \, \ip{\tphi_k}{\phi_j} \, \psi_j} \ip{\phi_k}{g} \plus
      \sum_{k=0}^{n-1} \, \ip{\phi_m}{\psi_k} \, \ip{\phi_k}{g}
      \quad\text{\footnotesize (by equation \eqref{tphij_eq})}
      \notag \\
& \Eq \ip{\phi_m}{g} \Minus \sum_{k=n}^{\infty} \sum_{j=0}^{n-1} \,
      \ip{\phi_m}{\psi_j} \, \ip{\phi_j}{\tphi_k} \, \ip{\phi_k}{g} \plus
      \sum_{k=0}^{n-1} \, \ip{\phi_m}{\psi_k} \, \ip{\phi_k}{g}.
      \label{calculation_eq}
\end{alignat}
Consequently, if we let $\vb_m$ be as in
Lemma \ref{linear independence completeness} and let
\[
\ub
\Eq \begin{bmatrix}
    \ip{\phi_0}{g} \\ \vdots \\ \ip{\phi_{n-1}}{g}
    \end{bmatrix}
\qquad\text{and}\qquad
\wb_k
\Eq \begin{bmatrix}
    \ip{\phi_0}{\tphi_k} \, \ip{\phi_k}{g} \\ \vdots \\
    \ip{\phi_{n-1}}{\tphi_k} \, \ip{\phi_k}{g}
    \end{bmatrix},
\]
then we have for all $m \ge n$ that
\begin{align*}
\ub \cdot \vb_m
\Eq \sum_{k=0}^{n-1} \, \ip{\phi_k}{g} \, \overline{\ip{\psi_k}{\phi_m}}
& \Eq \sum_{k=0}^{n-1} \, \ip{\phi_m}{\psi_k} \, \ip{\phi_k}{g}
      \\
& \Eq \sum_{k=n}^{\infty} \sum_{j=0}^{n-1} \,
      \ip{\phi_m}{\psi_j} \, \ip{\phi_j}{\tphi_k} \, \ip{\phi_k}{g}
      \qquad\text{\footnotesize (by equation \eqref{calculation_eq})}
      \\
& \Eq \lim_{N \to \infty} \sum_{k=n}^N \, (\wb_k \cdot \vb_m)
\Eq \lim_{N \to \infty} \biggparen{\sum_{k=n}^N \wb_k \tinyback} \cdot \vb_m
\end{align*}
But $\set{\vb_m}_{m \ge n}$ spans $\C^n$ by 
Lemma \ref{linear independence completeness},
so this implies that $\sum_{k=n}^N \wb_k$ converges weakly to $\ub$
in $\C^n$ as $N \to \infty.$
Since weak convergence implies strong convergence in finite-dimensional
normed spaces, it follows that
$\ub \Eq \sum_{k=n}^\infty \wb_k,$
with convergence in the norm of $\C^n$. We conclude that $\ub \Eq \sum_{k=n}^\infty \wb_k,$ and therefore
\begin{equation} \label{phikg_eq}
\ip{\phi_k}{g}
\Eq \sum_{j=n}^\infty \, \ip{\phi_k}{\tphi_j} \, \ip{\phi_j}{g},
\qquad\text{for } k = 0, \dots, n-1.
\end{equation}

Finally, in order to show that $\set{\phi_k}_{k \geq n}$ is a Schauder basis,
fix any vectors $f$ and $g$ in $\Hilbert.$
Then,
\begin{alignat*}{2}
\ip{f}{g}
& \Eq \sum_{k=0}^{n-1} \, \ip{f}{\psi_k} \, \ip{\phi_k}{g} \plus
      \sum_{k=n}^{\infty} \, \ip{f}{\psi_k} \, \ip{\phi_k}{g}
      && \qquad\text{\footnotesize (reproducing property)}
      \\
& \Eq \sum_{k=0}^{n-1} \, \ip{f}{\psi_k} \,
      \sum_{j=n}^\infty \, \ip{\phi_k}{\tphi_j} \, \ip{\phi_j}{g} \plus
      \sum_{k=n}^{\infty} \, \ip{f}{\psi_k} \, \ip{\phi_k}{g}
      && \qquad\text{\footnotesize (by equation \eqref{phikg_eq})}
      \allowdisplaybreaks \\
& \Eq \sum_{j=n}^\infty \, \Bigip{f}{\sum_{k=0}^{n-1} \,
      \ip{\tphi_j}{\phi_k} \, \psi_k} \ip{\phi_j}{g} \plus
      \sum_{k=n}^{\infty} \, \ip{f}{\psi_k} \, \ip{\phi_k}{g}
      \allowdisplaybreaks \\
& \Eq \sum_{j=n}^\infty \, \ip{f}{\tphi_j - \psi_j} \, \ip{\phi_j}{g} \plus
      \sum_{k=n}^{\infty} \, \ip{f}{\psi_k} \, \ip{\phi_k}{g}
      && \qquad\text{\footnotesize (by equation \eqref{tphij_eq})}
      \\
& \Eq \sum_{k=n}^\infty \, \ip{f}{\tphi_k} \, \ip{\phi_k}{g}.
\end{alignat*}
Lemma \ref{exact factorization implies schauder basis}
therefore implies that $\Phi'$ is a Schauder basis for $\Hilbert.$
\end{proof}

\printbibliography

\end{document}